\documentclass[11pt, a4paper]{amsart}
\usepackage{amssymb,amsmath}
\usepackage{amsthm, dsfont, a4}
\usepackage{hyperref}
\usepackage[all]{xy}
\usepackage[active]{srcltx}
\usepackage[short,nodayofweek]{datetime}
\newcommand\comment[1]{}

\def\GG{{\mathbb G}}
\def\NN{{\mathbb N}}
\def\ZZ{{\mathbb Z}}
\def\QQ{{\mathbb Q}}

\def\G{{\mathcal G}}

\def\m{{\mathfrak m}}

\def\tR{\tilde{R}}

\def\id{{\rm id}}

\def\cat#1{{\sf #1}}

\def\vect#1{\text{\boldmath $#1$\unboldmath}} 

\def\isom{\cong}

\def\ideal{\unlhd}
\def\gener#1{\left\langle #1 \right\rangle}
\renewcommand{\th}[1]{\theta^{(#1)}}

\DeclareMathOperator{\Ker}{Ker}
\DeclareMathOperator{\Ima}{Im}
\DeclareMathOperator{\Coker}{Coker}

\DeclareMathOperator{\Aut}{Aut}
\DeclareMathOperator{\GL}{GL}
\DeclareMathOperator{\Mat}{Mat}
\DeclareMathOperator{\Gal}{Gal}

\DeclareMathOperator{\Quot}{Quot}
\DeclareMathOperator{\Spec}{Spec}
\newcommand{\uGal}{\underline{\Gal}}

\def\markdef{\bf }

\theoremstyle{plain}
\newtheorem{thm}{Theorem}[section]

\newtheorem{prop}[thm]{Proposition}

\theoremstyle{definition}
\newtheorem{defn}[thm]{Definition}
\newtheorem{exmp}[thm]{Example}
\newtheorem{rem}[thm]{Remark}

\newtheoremstyle{Acknowledgements}
  {}
    {}
     {}
     {}
    {\bfseries}
    {}
     {.5em}
    {\thmname{#1}\thmnumber{ }\thmnote{ (#3)}}
\theoremstyle{Acknowledgements}

\begin{document}

\title[Non-free ID-modules]{Non-free iterative differential modules}

\author{Andreas Maurischat}
\address{\rm {\bf Andreas Maurischat}, Lehrstuhl A f\"ur Mathematik, RWTH Aachen University, Germany }
\email{\sf andreas.maurischat@matha.rwth-aachen.de}
\thanks{}

\subjclass[2010]{12H20, 13B05}

\keywords{Differential Galois theory, iterative derivations}

\date{\today}

\begin{abstract}
In \cite{am:pvtdsr} we established a Picard-Vessiot theory over differentially simple rings which may not be fields. Differential modules over such rings were proven to be locally free but don't have to be free as modules. In this article, we give a family of examples of non-free differential modules, and compute Picard-Vessiot rings as well as Galois groups for them.
\end{abstract}

\maketitle

\section{Introduction}

Differential Galois theory -- and also difference Galois theory -- is a generalisation of classical Galois theory to transcendental extensions. Instead of polynomial equations, one considers differential resp. difference equations, and even iterative differential equations in positive characteristic.
One branch of differential Galois theory is Picard-Vessiot theory, the study of linear differential equations. In this case the Galois group turns out to be a linear algebraic group or more general an affine group scheme of finite type over the field of constants. 
Originally, one considered extension of fields as in the finite Galois theory.
However, this had to be extended for two main reasons: In difference Galois theory, zero-divisors may occur in minimal solution rings. The Galois group scheme does not act algebraically on the solution field but on the Picard-Vessiot ring, an important subring whose field of fractions is the solution field.\\
Therefore, it is natural also to replace the base field by a base ring with ``nice'' properties. This has been done in several settings (see e.g.~\cite{ya:dnctgdd}, \cite{ka-am:pveasma}, \cite{am:pvtdsr}) where the base ring is a simple ring, i.e.~has no nontrivial ideals stable under the extra structure. In \cite{am:capvt}, we even gave an abstract setting covering all existing Picard-Vessiot theories, the base ring being again a simple object.

Having a base ring instead of a field, the first point is that not all modules have to be free.
In \cite{am:pvtdsr}, it has been shown that all iterative differential modules (ID-modules) over a simple iterative differential ring are locally free which is enough to obtain all the machinery of Picard-Vessiot-rings and Galois groups.\footnote{Local freeness of the modules have also been shown for other settings e.g.~in \cite{ya:dnctgdd} or \cite{ka-am:pveasma}.}
But this raises the question whether there really exist ID-modules which are not free.\\
In this article, we answer this question by providing a family of ID-modules which are not free as modules. We also compute Picard-Vessiot rings and Galois groups for them. As we don't assume that the field of constants is algebraically closed, Picard-Vessiot rings for a fixed module are not unique, and our example will also show this effect.
As in characteristic zero simple iterative differential rings are the same as simple differential rings this also provides examples for differential modules in characteristic zero which are not free.

\smallskip

The article is organized as follows.\\
In Section 2, the  basic notation, and some basic examples of ID-rings are given. We proceed in Section 3 with recalling some properties of ID-simple rings and ID-modules over ID-simple rings. In our definition, ID-modules are finitely generated as modules.
Section 4 is dedicated to Picard-Vessiot rings for ID-modules and their Galois groups.
Finally in Section 5, we give an example of a non-free ID-module over some ID-simple ring, and compute some of its Picard-Vessiot rings as well as the corresponding Galois groups.

\section{Basic notation}\label{sec:notation}

All rings are assumed to be commutative with unit and different from $\{0\}$.
We will use the following notation (see also \cite{am:igsidgg}). An iterative 
derivation on a ring $R$ is a family of additive maps $(\th{n})_{n\in \NN}$ on $R$ satisfying
\begin{enumerate}
\item $\th{0}=\id_R$,
\item\label{item:leibniz-rule} $\th{n}(rs)=\sum_{i+j=n}\th{i}(r)\th{j}(s)$ for all $r,s\in R$, $n\in \NN$, as well as
\item\label{item:iteration-rule} $\th{i}\circ \th{j}=\binom{i+j}{i}\th{i+j}$ for all $i,j\in \NN$.
\end{enumerate}
Most times we will consider the map
$$\theta:R\to R[[T]], r\mapsto \sum_{n=0}^\infty \th{n}(r)T^n,$$
where $R[[T]]$ is the power series ring over $R$ in one variable $T$.\\
The conditions that the $\th{n}$ are additive, and condition \ref{item:leibniz-rule} are then equivalent to $\theta$ being a homomorphism of rings, and condition \ref{item:iteration-rule} is equivalent to the commutativity of the diagram
\centerline{\xymatrix@+10pt{
R \ar[r]^{\theta_U} \ar[d]_{\theta} & R[[U]]
\ar[d]^{U\mapsto U+T} \\
R[[T]] \ar[r]^{\theta_U[[T]]} & R[[U,T]],
}}
where $\theta_U$ is the map $\theta$ with $T$ replaced by $U$ and $\theta_U[[T]]$ is the $T$-linear extension of $\theta_U$.

The pair $(R, (\th{n})_{n\in \NN})$ or the pair $(R,\theta)$ is then called an {\markdef ID-ring} and
$$C_R:=\{ r \in R\mid \theta(r)=r\}=\{r \in R\mid \th{n}(r)=0\, \forall\, n>0\}$$ is called the {\markdef ring of
  constants } of $(R,\theta)$. An ideal $I\ideal R$ is called an
 {\markdef ID-ideal} if $\theta(I)\subseteq I[[T]]$ and $R$ is
 {\markdef ID-simple} if $R$ has no
 ID-ideals apart from $\{0\}$ and $R$. An ID-ring which is a field is called an {\markdef ID-field}.
Iterative derivations are extended to localisations by
 $\theta(\frac{r}{s}):=\theta(r)\theta(s)^{-1}$ and to tensor products
 by 
$$\theta^{(k)}(r\otimes s)=\sum_{i+j=k} \theta^{(i)}(r)\otimes
\theta^{(j)}(s)$$ 
for all $k\geq 0$.

A homomorphism of ID-rings $f:S\to R$ is a ring homomorphism $f:S\to R$ s.t. $\theta_R^{(n)}\circ f=f\circ \theta_S^{(n)}$ for all $n\geq 0$.

\medskip

An {\markdef ID-module} $(M,\theta_M)$ over an ID-ring $R$ is a finitely generated $R$-module $M$ together with an iterative derivation $\theta_M$ on $M$, i.e.~an additive map $\theta_M:M\to M[[T]]$ such that $\theta_M(rm)=\theta(r)\theta_M(m)$, $\theta_M^{(0)}=\id_M$ and
$\theta_M^{(i)}\circ \theta_M^{(j)}=\binom{i+j}{i}\theta_M^{(i+j)}$ for all $i,j\geq 0$. 

A subset $N\subseteq M$ of an ID-module $(M,\theta_M)$ is called {\markdef ID-stable}, if $\theta_M^{(n)}(N)\subseteq N$ for all $n\geq 0$. An {\markdef ID-submodule} of $(M,\theta_M)$ is an ID-stable $R$-submodule $N$ of $M$ which is finitely generated as $R$-module.\footnote{If $R$ is an ID-simple ring, then ID-stable submodules are always ID-submodules.}
For an ID-module $(M,\theta_M)$ and an ID-stable $R$-submodule $N\subseteq M$, the factor module $M/N$ is again an ID-module with the induced iterative derivation.

The free $R$-module $R^n$ is an example of an ID-module over $R$ with  iterative derivation given componentwise. An ID-module $(M,\theta_M)$ over $R$ is called {\markdef trivial} if $M\isom R^n$ as ID-modules, i.e.~if $M$ has a basis of constant elements.

For ID-modules $(M,\theta_M)$, $(N,\theta_N)$, the {\markdef direct sum} $M\oplus N$ is an ID-module with iterative derivation given componentwise, and the {\markdef tensor product} $M\otimes_R N$ is an ID-module with iterative derivation $\theta_\otimes$ given by $\theta_\otimes^{(k)}(m\otimes n):=\sum_{i+j=k} \theta_M^{(i)}(m)\otimes \theta_N^{(j)}(n)$  for all $k\geq 0$.

For ID-modules $(M,\theta_M)$, $(N,\theta_N)$, a {\markdef morphism} $f:(M,\theta_M)\to(N,\theta_N)$ of ID-modules is a homomorphism $f:M\to N$ of the underlying modules such that $\theta_N^{(k)}\circ f=f\circ \theta_M^{(k)}$ for all $k\geq 0$. For a morphism $f:(M,\theta_M)\to(N,\theta_N)$,  the kernel $\Ker(f)$ and the image $\Ima(f)$ are ID-stable $R$-submodules of $M$ resp.~$N$. The image $\Ima(f)$ is indeed an ID-submodule, since it is isomorphic to $M/\Ker(f)$. Also the cokernel $\Coker(f)$ is an ID-module.

\begin{exmp}\label{ex:ID-rings}
\begin{enumerate}
\item For any field $C$ and $R:=C[t]$, the homomorphism of $C$-algebras
$\theta_t:R \to R[[T]]$
given by $\theta_t(t):=t+T$ is an iterative derivation on $R$ with field of constants $C$. This
iterative derivation will be called the {\markdef iterative derivation with respect to~$t$}. $R$ is indeed an ID-simple ring, since for any polynomial $0 \ne f\in R$ of degree $n$, $\theta^{(n)}(f)$ equals the leading coefficient of $f$, and hence is invertible in $R=C[t]$.
\item\label{item:der by t} For any field $C$,  $C[[t]]$ also is an ID-ring with the iterative derivation with respect to $t$, given by $\theta_t(f(t)):=f(t+T)$ for $f\in C[[t]]$. The constants of $(C[[t]],\theta_t)$ are $C$, and $(C[[t]],\theta_t)$ also is ID-simple, since for 
$f=\sum_{i=n}^\infty a_i t^i\in C[[t]]$ with $a_n\ne 0$, one has
$$\theta_t^{(n)}(f)= \sum_{i=n}^\infty a_i \binom{i}{n} t^{i-n}\in C[[t]]^\times.$$
Hence, every non-zero ID-ideal contains a unit.
This ID-ring will play an important role, since every ID-simple ring can be ID-embedded into $C[[t]]$ for an appropriate field $C$.
\item For any ring $R$, there is the {\markdef trivial} iterative derivation on $R$ given by
$\theta_0:R\to R[[T]],r\mapsto r\cdot T^0$. Obviously, the ring of constants of $(R,\theta_0)$ is
$R$ itself.
\item Given a differential ring $(R,\partial)$ containing the rationals (i.e.~a $\QQ$-algebra $R$ with a derivation $\partial$), then $\th{n}:=\frac{1}{n!}\partial^n$ defines an iterative derivation on $R$. On the other hand, for an iterative derivation $\theta$, the map $\partial:=\theta^{(1)}$ always is a derivation, and $\th{n}$ equals $\frac{1}{n!}\partial^n$ by the iteration rule. Hence, differential rings containing $\QQ$ are special cases of ID-rings.\\
Since for a differentially simple ring in characteristic zero, its ring of constants always is a field (same proof as for ID-simple rings), we see that differentially simple rings in characteristic zero are a special case of ID-simple rings in arbitrary characteristic.
\end{enumerate}
\end{exmp}

\section{Properties of ID-simple rings and ID-modules over ID-simple rings}

We summarize some properties of ID-simple rings. Proofs can be found in  \cite{bhm-mvdp:ideac} or  \cite{am:pvtdsr}.

Throughout the section, let $(S,\theta)$ denote an ID-simple ring with constants $C$.

Then $S$ is an integral domain, and its ring of constants $C$ is a field. Furthermore, the field of fractions of $S$ has the same constants as $S$ (cf.~\cite[Lemma 3.2]{bhm-mvdp:ideac}).

If $\m\ideal S$ is a maximal ideal, and $k:=S/\m$ is the residue field. Then $(S,\theta)$ can be embedded into $(k[[t]],\theta_t)$ as ID-rings via
$$S\to k[[t]], s\mapsto \sum_{n=0}^\infty \overline{\th{n}(s)}t^n,$$
where $\overline{\th{n}(s)}$ denotes the image of $\th{n}(s)\in S$ in the residue field $k$
(cf.~\cite[Thm.~3.4]{am:pvtdsr}). This will be important later on, in the case that $k=C$.

\medskip

Now consider an ID-module $(M,\theta_M)$ over the ID-simple ring $(S,\theta)$.
Such an ID-module is projective as $S$-module, i.e.~locally free (cf.~\cite[Thm.~4.3]{am:pvtdsr}). In particular, if $S$ is a local ring, the ID-module is free as a module.

For the special ID-simple ring $(C[[t]],\theta_t)$, we even have:
\begin{thm}\label{thm:trivial-over-ct} (cf.~\cite[Thm.~4.6]{am:pvtdsr})\\
Let $C$ be a field. Then every ID-module over $(C[[t]],\theta_t)$ is trivial.
\end{thm}

\section{Picard-Vessiot rings and Galois groups}

Throughout the section, let $(S,\theta)$ denote an ID-simple ring, and $(M,\theta_M)$ an ID-module over $S$.

\begin{defn}\label{def:pv-ring}

A {\markdef solution ring} for $M$ is an ID-ring $0\ne (R,\theta_R)$ together with a homomorphism of ID-rings $f:S\to R$ s.t.~the natural homomorphism 
\begin{equation}\label{eq:natural-homom}
R\otimes_{C_R} C_{R\otimes_S M}\longrightarrow R\otimes_S M \tag{$\dagger$}
\end{equation} 
 is an isomorphism.\footnote{Take care that the definition of a solution ring given here is different to that in \cite{am:pvtdsr}, but consistent with the one in \cite[Def.~5.3]{am:capvt}. However, the definitions of a PV-ring are all equivalent.}
A {\markdef Picard-Vessiot ring} (PV-ring) for $M$ is a \emph{minimal} solution ring $0\ne (R,\theta_R)$ which is ID-simple and has the same field of constants as $S$.
Here, \emph{minimal} means that if $0\ne (\tilde{R},\theta_{\tilde{R}})$ with $\tilde{f}:S\to \tilde{R}$ is another solution ring, then any monic ID-homomorphism $g:\tilde{R}\to R$ (if it exists) is an isomorphism.
\end{defn}

\begin{rem}
\begin{enumerate}
\item Since the kernel of an ID-homomorphism is an ID-ideal, and $S$ is ID-simple, the homomorphism $f$ is always injective. Therefore, we can view any solution ring $R$ as an extension of $S$, and we will omit the homomorphism $f$.
\item If $R$ is an ID-simple ring, then the homomorphism (\ref{eq:natural-homom}) is always injective (cf.~\cite[Prop.~4.6(iii)]{am:capvt}), and therefore $R\otimes_{C_R} C_{R\otimes_S M}$ can be seen as a free ID-submodule of $R\otimes_S M$. Hence, an ID-simple ring is a solution ring if and only if $R\otimes_S M$ has a basis of ID-constant elements.
\item Assume that $M$ is a free $S$-module with basis $\vect{b}=(b_1,\dots, b_r)$, and $R$ is an ID-simple solution ring for $M$,  then there is a matrix $X\in \GL_r(R)$ s.t. $\vect{b}X$ is a basis of constant elements in $R\otimes_S M$. Such a matrix will be called a {\markdef fundamental solution matrix} for $M$ (with respect to $\vect{b}$).
\item Every PV-ring is faithfully flat over $S$ (cf.~\cite[Cor.~5.5]{am:pvtdsr}). This will have a nice effect on the Galois groups. Namely, the Galois group of $R/S$ (definition given later) is the same as the one of $\Quot(S)\otimes_S R$ over the field of fractions $\Quot(S)$. 
\item If $\m\ideal S$ is a maximal ideal, $k:=S/\m$, and $(S,\theta)\to (k[[t]],\theta_t)$ the embedding given above. Then $k[[t]]$ is a solution ring for any ID-module $M$ over $S$.\\
The next proposition shows the importance of this remark in case that $k=C$. In this case the ring $\tR$ of the proposition can be chosen to be the ring $C[[t]]$.
\end{enumerate}
\end{rem}

\begin{prop}\label{prop:pv-ring-exists-inside-simple-sol-ring}
Let $\tR$ be an ID-simple solution ring for $M$ with the same constants as $S$ (i.e.~$C_{\tR}=C_S$). Then there is a unique Picard-Vessiot ring $R$ inside $\tR$.
\begin{enumerate}
\item If $M$ is a free $S$-module, then $R$ is the $S$-subalgebra of $\tR$ generated by the coefficients of a fundamental solution matrix and the inverse of its determinant.
\item For general $M$ (i.e.~$M$ ``only'' locally free), $R$ is obtained as follows:\\
Let $x_1,\dots, x_l\in S$ such that $\gener{x_1,\dots, x_l}_S=S$ and $M[\frac{1}{x_i}]$ is free over $S[\frac{1}{x_i}]$ for all $i=1,\dots, l$. Let $\vect{e}=(e_1,\dots, e_r)$ be a $\tR$-basis of $\tR\otimes_S M$ consisting of ID-constant elements, and for all $i$, let $\vect{b_i}$ be a basis of $M[\frac{1}{x_i}]$ over $S[\frac{1}{x_i}]$ consisting of elements in $M$.
Define the matrices $Y_i\in \Mat_{r\times r}(\tR)$ via $\vect{b_i}=\vect{e}Y_i$ ($i=1,\dots,l$), and choose $n_i\in\NN$ such that $x_i^{n_i}M\subseteq \gener{\vect{b_i}}_S$. Then $$R:=S[Y_j, \det(x_j^{n_j}Y_j^{-1})\mid j=1,\dots l].$$
\end{enumerate}
In particular, every Picard-Vessiot ring for $M$ is a finitely generated $S$-algebra.
\end{prop}

For the proofs see \cite[Prop.~5.3 and Thm.~5.4]{am:pvtdsr}.

\begin{rem}
\begin{enumerate}
\item The first part of the proposition shows that for an ID-module which is free as $S$-module, the definition given here coincides with the usual one given for example in \cite[Sect.~3]{bhm-mvdp:ideac}  (if the constants are algebraically closed) resp.~in \cite[Def.~2.3]{am:igsidgg}. The second part gives a receipt to compute PV-rings in general.
\item If $S$ has a maximal ideal $\m\ideal S$ such that $S/\m=C$, then $(C[[t]],\theta_t)$ is an ID-simple solution ring for $M$ with same constants as $S$. Hence, the existence of a Picard-Vessiot ring is guaranteed in this case.
\item Different maximal ideals of $S$ with residue field $C$ lead to different ID-embeddings $S\to C[[t]]$. If the field of constants is not algebraically closed, this can lead to different non-isomorphic PV-rings, as in the example given in Section \ref{sec:example}.
\item If the field of constants is algebraically closed, then as in the case of differential fields, for every ID-module there exists a PV-ring, and furthermore, all PV-rings for a given ID-module are isomorphic (cf.~\cite[Thm.~5.6]{am:pvtdsr}).
\item Picard-Vessiot rings behave well under extension of the base field by constants. By this we mean the following: Let $D/C$ be a field extension, $D$ equipped with the trivial iterative derivation, and $S_D:=S\otimes_C D$. If $R/S$ is a Picard-Vessiot ring for some ID-module $M$, then $R_D:=R\otimes_C D$ is a 
Picard-Vessiot ring for the $S_D$-ID-module $M \otimes_C D$.\\
In view of the previous remark (with $D$ being the algebraic closure of $C$) this also implies that all Picard-Vessiot rings for some ID-module become isomorphic over the algebraic closure of the constants $C$.
\end{enumerate}
\end{rem}

\bigskip

\subsection*{The differential Galois group scheme}

In the classical case over a differential field with algebraically closed field of constants the differential Galois group is defined as the group of differential automorphisms of a Picard-Vessiot ring over the base differential field. This group turns out to be a Zariski-closed subgroup of some $\GL_n(C)$ where $C$ denotes the field of constants.\\
If the constants are not algebraically closed, and in the iterative differential setting in positive characteristic, this group of differential automorphisms might be ``too small'', and one has to consider a group valued functor instead.

\begin{defn}
Let $S$ be an ID-simple ring with field of constants $C$, and let $R$ be a PV-ring for some ID-module $M$ over $S$. Then we define the group functor
$\underline{\Aut}^{\theta}(R/S):\cat{Algebras}/C\to \cat{Grps}$ associating to each $C$-algebra $D$ the group $\Aut^\theta(R\otimes_C D/S\otimes_C D)$ of ID-automorphisms of $R\otimes_C D$ that fix the elements of $S\otimes_C D$. Here, $D$ is equipped with the trivial iterative derivation.\\
We call $\G=\underline{\Aut}^{\theta}(R/S)$ the {\markdef ID-Galois group (scheme)} of $R/S$ and also denote it by $\uGal(R/S)$.
\end{defn}

The term ``Galois group scheme'' is justified by the following

\begin{thm} (cf.~\cite[Cor.~5.4]{am:pvtdsr})\\
The group functor $\uGal(R/S)$ is represented by the algebra $C_{R\otimes_S R}$ which is a finitely generated $C$-algebra. Hence, $\uGal(R/S)=\Spec(C_{R\otimes_S R})$ is an affine group scheme of finite type over $C$.
\end{thm}

\begin{rem}
\begin{enumerate}
\item \label{item:gen-of-H-in-free-case} If $R$ is a PV-ring for a free module $M$ of rank $n$, hence generated by the entries of a fundamental solution matrix $X\in \GL_n(R)$ and the entries of $X^{-1}$, then the ring of constants $C_{R\otimes_S R}$ is generated by the entries of $Z:=(X^{-1}\otimes 1)(1 \otimes X)\in \GL_n(C_{R\otimes_S R})$ \footnote{By $1\otimes X$ we mean the matrix whose $(i,j)$-th entry is $1\otimes X_{ij}$, and similar for $X^{-1}\otimes 1$.} and of $Z^{-1}=(1\otimes X^{-1})(X\otimes 1)$.\\
This representation of $C_{R\otimes_S R}$ provides a closed embedding $\uGal(R/S)
\hookrightarrow \GL_n$, and the action of $\uGal(R/S)$ on $R$ is given by multiplying the fundamental solution matrix $X$ by a matrix of $\uGal(R/S)\subseteq \GL_n$ from the right.
\item As a Picard-Vessiot ring $R$ is faithfully flat over $S$ (cf.~\cite[Cor.~5.5]{am:pvtdsr}), it is linearly disjoint over $S$ to the field of fractions $F=\Quot(S)$. In particular, $R_F=F\otimes_S R$ is the localisation of $R$ by all non-zero elements in $S$, and hence is a Picard-Vessiot ring over $F$. Therefore,
$\uGal(R_F/F)=\uGal(R/S)$. Hence, the ID-Galois group can be computed over any localisation of $S$. This simplifies its computation, since one can reduce to the case of a free module as given in \ref{item:gen-of-H-in-free-case}.
\end{enumerate}
\end{rem}

\section{Example}\label{sec:example}

We now give an example of an ID-module which is not free as a module. Therefore, we first need an ID-simple ring for which non-free projective modules exist. The most standard examples for non-free projective modules are non-principle ideals of  Dedekind domains. This will be our example after having attached iterative derivations to the Dedekind domain as well as the module.

\subsection*{An ID-simple ring having non-free projective modules}

Let $C$ be any field and let
$$S:=C[s,t,\frac{1}{3s^2-1}]/(s^3-s-t^2),$$
which is the localisation of an integral extension of $C[t]$ of degree $3$. $S$ is integrally closed, and hence $S$ is a Dedekind domain.

Since we inverted $3s^2-1$, $S$ is \'e{}tale over $C[t]$, and hence the iterative derivation $\theta_t$ by $t$ on $C[t]$ can be uniquely extended to an iterative derivation $\theta$ on $S$ (cf.~\cite[Thm.~27.2]{hm:crt}). The $\th{n}(s)$ can be computed successively using the equation
$$\theta(s)^3-\theta(s)=\theta(t)^2=(t+T)^2,$$
obtained from $s^3-s=t^2$ by applying $\theta$.
In particular,
$$\theta^{(1)}(s)=\frac{2t}{3s^2-1}.$$

\begin{prop}
The ID-ring $(S,\theta)$ is ID-simple.
\end{prop}

\begin{proof}
If $0\ne I\lneq S$ is an ideal, then $I\cap C[t]$ is an ideal of $C[t]$. Since, $S$ is the localisation of an integral extension, the ideal $I\cap C[t]$ also is nontrivial. Furthermore, if $I$ is an ID-ideal, then obviously $I\cap C[t]$ is also ID-stable, hence an ID-ideal of $C[t]$. But $(C[t],\theta_t)$ is ID-simple by example \ref{ex:ID-rings}. Hence, $S$ also contains no nontrivial ID-ideals.
\end{proof}

\subsection*{A non-free ID-module over $S$ in characteristic zero}

We first restrict to the case of ${\rm char}(C)=0$. In this case, an iterative derivation $\theta_M$ on $M$ is uniquely determined by the derivation $\partial_M:=\theta_M^{(1)}$.

We consider the $S$-module $M$ generated by two elements $f_1$ and $f_2$ subject to the relations $tf_1-sf_2=0$ and $(s^2-1)f_1-tf_2=0$. As $S$-module $M$ is isomorphic to the ideal $I=\gener{s,t}_S\subseteq S$ by mapping $f_1$ to $s$ and $f_2$ to $t$.
Since $I$ is a non-principal ideal of $S$, $I$ and hence $M$ is a non-free projective $S$-module of rank $1$.

\begin{thm}
For any $b\in S$,
$$\partial_M(f_1):=bf_1+\frac{3s^2+1}{3s^2-1}f_2\quad \text{ and } \quad
\partial_M(f_2):=sf_1+bf_2$$
defines a derivation on $M$.\\
Furthermore, every derivation on $M$ can be written in this form.
\end{thm}

\begin{proof}
Using the definition, one obtains
\begin{eqnarray*}
\partial_M(tf_1-sf_2) &=& \partial(t)f_1+t\partial_M(f_1)- \partial(s)f_2-s\partial_M(f_2) \\
&=& f_1+tbf_1+t\frac{3s^2+1}{3s^2-1}f_2 - \frac{2t}{3s^2-1}f_2-s^2f_1-sbf_2 \\
&=& b(tf_1-sf_2)+(1-s^2)f_1+\left( \frac{3s^2+1}{3s^2-1}-\frac{2}{3s^2-1}\right)tf_2\\
&=&  b(tf_1-sf_2)-\left((s^2-1)f_1-tf_2\right)=0,
\end{eqnarray*}
and similarly $\partial_M\left((s^2-1)f_1-tf_2\right)=0$. Hence, the derivation is a well-defined derivation on $M$.

On the other hand, given a derivation $\partial_M$ on $M$, we obtain a derivation on the $\Quot(S)$-vector space $\tilde{M}:=\Quot(S)\otimes_S M$ by scalar extension. The element $f_2$ is a basis of that vector space, and $f_1=\frac{s}{t}f_2\in \tilde{M}$.\\
Hence, $\partial_M(f_2)=af_2$ for some $a\in \Quot(S)$ which can also be written as
$\partial_M(f_2)=sf_1+bf_2$ for $b=a-\frac{s^2}{t}$.\\
Then 
\begin{eqnarray*}
\partial_M(f_1) &=& \partial_M\left( \frac{s}{t}f_2\right)= \partial\left( \frac{s}{t}\right)f_2+
\frac{s}{t}\partial_M(f_2)
=\frac{\frac{2t}{3s^2-1}t-s}{t^2}f_2+\frac{s}{t}(sf_1+bf_2)\\
&=&\left(\frac{2}{3s^2-1}-\frac{s}{t^2}\right)f_2+ \frac{s^3}{t^2}f_2+bf_1
= bf_1+ \frac{3s^2+1}{3s^2-1}f_2.
\end{eqnarray*}
Therefore, $\partial_M$ is of the form above for some $b\in \Quot(S)$. But, $M$ is stable under this derivation if and only if $bf_2\in M$ as well as $bf_1\in M$. So $M$ is stable under the derivation if and only if $bM\subseteq M$, i.e.~$b\in S$.
\end{proof}

\subsection*{Picard-Vessiot rings and Galois groups for this ID-module}

The ID-ring $S$ has a $C$-rational point, e.g.~the ideal $\m=(s-1,t)$, and we obtain an ID-embedding $S\to (S/\m)[[t]]\isom C[[t]]$.\footnote{Using the variable $t$ in the power series ring is justified by the fact, that $t\in S$ indeed maps to $t\in C[[t]]$ via the given embedding.} So by Prop.~\ref{prop:pv-ring-exists-inside-simple-sol-ring} there exists a Picard-Vessiot ring for $M$ inside $C[[t]]$, and we follow the explicit description of the Picard-Vessiot ring given there.

First at all, we choose $x_1:=s$ and $x_2:=s^2-1$. Then $M[\frac{1}{x_1}]$ is free over $S[\frac{1}{x_1}]$ with basis $b_1:=f_1$, and $M[\frac{1}{x_2}]$ is free over $S[\frac{1}{x_2}]$ with basis $b_2:=f_2$. Further, $x_1M=sM\subseteq \gener{b_1}_S$ and $x_2M=(s^2-1)M\subseteq \gener{b_2}_S$, hence we can choose $n_1=n_2=1$.

Let $0\ne e\in C[[t]]\otimes_S M$ be a constant element, and $y\in C[[t]]$ such that $f_1=ye$.
As $s\not\in \m$, $s$ is invertible in $S_\m\isom C[[t]]$, and $f_2=\frac{t}{s}f_1
\in C[[t]]\otimes_S M$. In particular, $f_1$ is a basis of $C[[t]]\otimes_S M$.
Actually, this also implies that $y$ is invertible in $C[[t]]$, as it is the base change matrix between the bases $f_1$ and $e$ of $C[[t]]\otimes_S M$. As 
$$\partial(y)e=\partial_M(ye)=\partial(f_1)=bf_1+\frac{3s^2+1}{3s^2-1}f_2=\left( b+\frac{3s^2+1}{3s^2-1}\frac{t}{s}\right)ye,$$
$y$ is a solution of the differential equation
\begin{equation}\label{eq:diff-eq-for-y}
\partial(y)=\left( b+\frac{3s^2+1}{3s^2-1}\frac{t}{s}\right)y.
\end{equation}

Furthermore we get (with notation as in Prop.~\ref{prop:pv-ring-exists-inside-simple-sol-ring}) 
$Y_1=y, Y_2=\frac{yt}{s}$, $\det(x_1^{n_1}Y_1^{-1})=\frac{s}{y}$ as well as
$\det(x_2^{n_2}Y_2^{-1})=\frac{(s^2-1)s}{yt}=\frac{t}{y}$. Hence,
$$R=S[y,\frac{yt}{s}, \frac{s}{y}, \frac{t}{y}].$$
Be aware that the inverse of $y$ is not in $R$.

\bigskip

As $M$ is a module of rank $1$, the Galois group is a subgroup of $\GL_{1}=\GG_{m}$. Hence, the Galois group is $\GG_{m}$ or one of the groups $\mu_n$ of $n$-th roots of unity. The Galois group is $\GG_{m}$ if $y$ is transcendental over $S$, and it is $\mu_n$ if $n$ is the least positive integer such that $y^n\in S$.

Whether $y$ is transcendental over $S$ or not, depends on the choice of $b$.
\begin{enumerate}
\item If we take, $b=\frac{-3st}{3s^2-1}$, then $\partial(y)=\frac{t}{(3s^2-1)s}y$, and hence
$$\partial\left(\frac{y^2}{s}\right)=\frac{2y\partial(y)}{s}-\frac{y^2\partial(s)}{s^2}=0.$$
Therefore, $\frac{y^2}{s}$ is a constant, i.e.~$y$ is a square root in $C[[t]]$ of $cs$ for some $0\ne c\in C$.
Actually, any $c\ne 0$ such that $cs$ is a square in $C[[t]]$ will do, as different choices just correspond to different choices of the constant basis $e$. As in $C[[t]]$, $s\equiv 1\mod t$, there exists a square root $\sqrt{s}\in C[[t]]$ of $s$ with $\sqrt{s}\equiv 1\mod t$. Hence, we can choose $c=1$, and $y=\sqrt{s}$, and obtain
$$R=S\left[\sqrt{s},\frac{t}{\sqrt{s}}\right],$$
an extension of degree $2$ and Galois group $\mu_2$.

If we would have taken the maximal ideal to be $\m=(s+1,t)$, and the corresponding embedding $S\hookrightarrow (S/\m)[[t]]\isom C[[t]]$, then in the last step $s\equiv -1\mod t$ inside $C[[t]]$, and we have a square root $\sqrt{-s}$ of $-s$ in $C[[t]]$ with $\sqrt{-s} \equiv 1\mod t$. This leads to the Picard-Vessiot ring
$$R_2=S\left[\sqrt{-s},\frac{t}{\sqrt{-s}}\right],$$
which is not isomorphic as an $S$-algebra to $R$ above, if $-1$ is not a square in $C^\times$.
The Galois group, however, is again $\mu_2$. 
\item If we take, $b=0$, then inside $\Quot(S)$ we have 
$$\partial\left(\frac{y}{s}\right)= \frac{\partial(y)}{s}-y\frac{\partial(s)}{s^2}
= \frac{(3s^2+1)t}{(3s^2-1)s}\frac{y}{s}-\frac{2t}{(3s^2-1)s}\frac{y}{s}=\frac{t}{s}\frac{y}{s}$$
If $\frac{y}{s}$ was not transcendental over $\Quot(S)$, then some $n$-th power $w=\left(\frac{y}{s}\right)^n$ would be in $\Quot(S)$. For $w$ we get the differential equation
$$\partial(w)= \frac{nt}{s} w.$$
Writing $w=w_0(s)+w_1(s)t$ with $w_0,w_1\in C(s)$, we calculate
\begin{eqnarray*}
\partial(w)&=& \partial(w_0(s))+\partial(w_1(s))t+w_1(s)
=w'_0(s)\frac{2t}{3s^2-1}+w'_1(s)\frac{2t}{3s^2-1}t+w_1(s)\\
&=& \left( w_1(s)+ w'_1(s)\frac{2(s^3-s)}{3s^2-1} \right)+ \frac{2w'_0(s)}{3s^2-1}t,
\end{eqnarray*}
as well as
$$ \frac{nt}{s} w= \frac{nt}{s}w_0(s)+ \frac{nt^2}{s}w_1(s)
= n(s^2-1)w_1(s) + \frac{nw_0(s)}{s}t.$$
Here $w'_0(s)$ and $w'_1(s)$ denote the usual derivatives of rational functions.
By comparing coefficients of $t$, we obtain
\begin{eqnarray*}
  \frac{nw_0(s)}{s} &=& \frac{2w'_0(s)}{3s^2-1} \qquad \text{and} \\
    (ns^2-n-1)w_1(s) &=& w'_1(s)\frac{2(s^3-s)}{3s^2-1}.
\end{eqnarray*}
If $w_0,w_1\ne 0$, this implies
$$\deg_s(w_0(s))=\deg_s\left(  \frac{nw_0(s)}{s}\right)+1
= \deg_s\left(  \frac{2w'_0(s)}{3s^2-1}\right) +1
=\deg_s(w'_0(s)) -1,$$
and
$$\deg_s(w_1(s))= \deg_s\left( (ns^2-n-1)w_1(s)\right) -2
=  \deg_s\left( w'_1(s)\frac{2(s^3-s)}{3s^2-1}\right) -2
=\deg_s(w'_1(s)) -1.$$
But $\deg_s(f'(s))\leq \deg_s(f(s))-1$ for all $0\ne f(s)\in C(s)$, and hence $w_0(s)=w_1(s)=0$, i.e.~$w=0$ which is impossible.\\
Hence, there is no such $w$, and $\frac{y}{s}$ and also $y$ are transcendental over $S$.
\end{enumerate}

\subsection*{A non-free ID-module over $S$ in positive characteristic}

Finding an example in positive characteristic is harder, since one is not done by giving just $\theta_M^{(1)}$, but by giving all $\theta_M^{(p^j)}$ which moreover have to commute and have to be nilpotent of order $p$.

We will follow a different approach here.
We start with the example in characteristic zero given in the previous paragraph.

The iterative derivation on $C[t]$ is already defined on $\ZZ[t]$ and extends to the ring
$S_\ZZ:=\ZZ[s,t,\frac{1}{3s^2-1}]/(s^3-s-t^2),$
since the latter is \'e{}tale over the former.

Therefore, the ID-ring $S$ from above (with constants $C$) is obtained as $S=C\otimes_\ZZ S_\ZZ$. And this holds in any characteristic.
For constructing an ID-module $M$ over $S$, one can start with a projective module $M'$ over $S_\ZZ$, and define a derivation on $M:=S\otimes_{S_\ZZ} M'$. If the corresponding iterative derivation stabilizes $M'$, one can reduce modulo $p$, to obtain an iterative derivation on $M'/pM'$.
This is then an ID-module over $\mathbb{F}_p\otimes_\ZZ S_\ZZ$.

Therefore take the ID-module over $S_\QQ$ from above with $b=\frac{-3st}{3s^2-1}$. Then we know that $e=\frac{1}{y}f_1$ is a constant basis of $R\otimes M$, where $y=\sqrt{s}$.

Hence, $\theta_M(f_1)=\theta_M(ye)=\theta(y)e=\frac{\theta(y)}{y}f_1$.
Replacing $y$ by $\sqrt{s}$ and using the chain rule (cf.~\cite[Prop.~7.2]{ar:icac})
one obtains:
$$\theta(\sqrt{s})=\theta_s(\sqrt{s})|_{T=\theta(s)-s}=(s+T)^{\frac{1}{2}}|_{T=\theta(s)-s}
=\sqrt{s}\cdot \sum_{k=0}^\infty \binom{1/2}{k} \left(\frac{\theta(s)}{s}-1\right)^k.$$
Therefore, all appearing rational numbers only have powers of $2$ in the denominator, and we can reduce modulo any prime $p$ different from $2$,
obtaining a non-free ID-module in characteristic $p$.


\bibliographystyle{plain}
\def\cprime{$'$}

\vspace*{.5cm}

\end{document}